\documentclass[10pt]{IEEEtran}

\usepackage{amsmath,amssymb,amsfonts}
\usepackage{times}
\usepackage{graphicx}
\usepackage{cite}
\usepackage{tabularx}
\usepackage{color}

\newlength{\figwidth}
\newlength{\figheight}
\newlength{\philwidth}

\def\mbb#1{\mathbb{#1}}

\def\lrs#1{\ensuremath{\left[{#1}\right]}}%

\renewcommand{\r}[1]{\ensuremath{\big|_{#1}}}

\newcommand{\norm}[1]{\ensuremath{\|#1\|}}

\newcommand{\E}{\ensuremath{\mathbb{E}}}

\newcommand{\R}{\ensuremath{\mathbb{R}}}
\newcommand{\C}{\ensuremath{\mathbb{C}}}

\newcommand{\N}{\ensuremath{\mathbb{N}}}

\def\Prblr#1{\mbb{P}\kern-2pt\lrs{{#1}}}%

\newtheorem{cor}{Corollary}
\newtheorem{remark}{Remark}
\newtheorem{prop}{Proposition}

\newtheorem{lemma}{Lemma}
\newtheorem{theorem}{Theorem}
\newtheorem{definition}{Definition}
\newtheorem{example}{Example}

\definecolor{darkorange}{rgb}{1.00,0.50,0.00}
\begin{document}

\title{A Random Matrix--Theoretic Approach to Handling Singular Covariance Estimates}
\author{Thomas L. Marzetta, Gabriel H. Tucci and Steven H. Simon
\thanks{
T. L. Marzetta and G. H. Tucci
  are with Bell Laboratories,
  Alcatel--Lucent, 600 Mountain Ave, Murray Hill, NJ 07974,
  e-mail: marzetta@alcatel-lucent.com, gabriel.tucci@alcatel-lucent.com. 
  Steven H. Simon is with University of Oxford, e-mail: s.simon1@physics.ox.ac.uk.}}

\maketitle

\begin{abstract}
In many practical situations we would like to estimate the covariance matrix of a set of variables from an
insufficient amount of data. More specifically, if we have a set of $N$ independent, identically distributed measurements of an
$M$ dimensional random vector the maximum likelihood estimate is the sample covariance matrix. Here we consider the case where $N<M$ such that this estimate is singular (non--invertible) and therefore fundamentally bad. We present a radically new 
approach to deal with this situation. Let $X$ be the $M\times N$ data matrix, where the columns are 
the $N$ independent realizations of the random vector with covariance matrix $\Sigma$. Without loss of generality, 
and for simplicity, we can assume that the random variables have zero mean. We would like 
to estimate $\Sigma$ from $X$. Let $K$ be the classical sample covariance matrix. Fix a parameter $1\leq L\leq N$ and consider an ensemble of $L\times M$ random unitary matrices, $\{\Phi\}$, having Haar probability measure (isotropically random). Pre-- and post--multiply $K$ by $\Phi$, and by the conjugate transpose of $\Phi$ respectively, to produce a non--singular $L\times L$ reduced dimension covariance estimate. A new estimate for $\Sigma$, denoted by $\mathrm{cov}_L(K)$, is obtained by a) projecting the reduced covariance estimate out (to $M\times M$) through pre-- and post--multiplication by the conjugate transpose of $\Phi$, and by $\Phi$ respectively, and b) taking the expectation over the unitary ensemble. Another new estimate (this time for $\Sigma^{-1}$), $\mathrm{invcov}_L(K)$, is obtained by a) inverting the reduced covariance estimate, b) projecting the inverse out (to $M\times M$) through pre-- and post--multiplication by the conjugate transpose of $\Phi$, and by $\Phi$ respectively, and c) taking the expectation over the unitary ensemble. We show that the estimate $\mathrm{cov}$ is equivalent to diagonal loading. Both estimates $\mathrm{invcov}$ and $\mathrm{cov}$ retain the original eigenvectors and make nonzero the formerly zero eigenvalues. We have a closed form analytical expression for $\mathrm{invcov}$ in terms of its eigenvector and eigenvalue decomposition. We motivate the use of $\mathrm{invcov}$ through applications to linear estimation, supervised learning, and high--resolution spectral estimation. We also compare the performance of the estimator $\mathrm{invcov}$ with respect to diagonal loading.

\end{abstract}


\section{Introduction}%

\par The estimation of a covariance matrix from an insufficient amount of
data is one of the most common multivariate problems in statistics, signal processing, and
learning theory. Inexpensive sensors permit ever more
measurements to be taken simultaneously. Thus the dimensions of
feature vectors are growing. But typically the number of independent
measurements of the feature vector are not increasing at a commensurate
rate. Consequently, for many problems, the sample covariance matrix is almost always singular (non--invertible). 
More precisely, given a set of independent multivariate Gaussian feature vectors, the
sample covariance matrix is a maximum likelihood estimate. When the
number of feature vectors is smaller than their dimension then the
estimate is singular, and the sample covariance is a fundamentally bad
estimate in the sense that the maximum likelihood principle yields a non--unique
estimate having infinite likelihood. The sample covariance finds
linear relations among the random variables when there may be
none. The estimates for the larger eigenvalues are typically too big,
and the estimates for the small eigenvalues are typically too small.

\par The conventional treatment of covariance singularity artificially
converts the singular sample covariance matrix into an invertible
(positive--definite) covariance by the simple expedient of adding a
positive diagonal matrix, or more generally, by taking a linear
combination of the sample covariance and an identity matrix. This
procedure is variously called ``diagonal loading'' or 
``ridge regression'' \cite{ridge}, \cite{Edelman2}. The resulting covariance has the same eigenvectors as the sample
covariance, and eigenvalues which are uniformly scaled and shifted versions of
the sample covariance eigenvalues. The method of Ledoit and Wolf
\cite{LW} automatically chooses the combining coefficients for diagonal
loading.

\par We propose a radically different alternative to diagonal loading which
is based on an ensemble of dimensionality reducing random unitary
matrices. The concept is that the unitary matrix multiplies the
feature vectors to produce shortened feature vectors, having dimension
significantly smaller than the number of feature vectors, which
produce a statistically meaningful and invertible covariance estimate. The
covariance estimate is used to compute an estimate for the ultimate
quantity or quantities of interest. Finally this estimate is averaged
over the ensemble of unitary matrices. We consider two versions of this scheme which we call $\mathrm{cov}$ and $\mathrm{incov}$. 
We show that the estimate $\mathrm{cov}$ is equivalent to diagonal loading. Both estimates $\mathrm{invcov}$ and $\mathrm{cov}$ retain the original eigenvectors and make nonzero the formerly zero eigenvalues. We have a closed form analytical expression for $\mathrm{invcov}$ in terms of its eigenvector and eigenvalue decomposition. We motivate the use of $\mathrm{invcov}$ through applications to linear estimation, supervised learning, and high--resolution spectral estimation. We also compare the performance of the estimator $\mathrm{invcov}$ with respect to diagonal loading.


\par Throughout the paper we will denote by $A^{*}$ the complex conjugate transpose of the matrix $A$. $I_{N}$ will represent the $N\times N$ 
identity matrix. We let $\mathrm{Tr}$ be the non--normalized trace for square matrices, defined by,
$$
\mathrm{Tr}(A):=\sum_{i=1}^{N}{a_{ii}},
$$
where $a_{ii}$ are the diagonal elements of the $N\times N$ matrix $A$. We also let $\mathrm{tr}$ be the normalized trace, defined by $\mathrm{tr}(A)=\frac{1}{N}\mathrm{Tr}(A)$.

\section{New Approach to Handling Covariance Singularity}

We begin with a set of $N$ independent identically distributed
measurements of an $M$ dimensional random vector where $N < M$. We
introduce an ensemble of $L \times M$ random unitary matrices, such
that $L \leq N$. The unitary matrix multiplies the feature vectors to produce a set of $N$ 
feature vectors of dimension $L$ from which we obtain an invertible sample covariance
matrix. The dimensionality reduction process is reversible (i.e., no
information is thrown away) provided it is done for a sufficient
multiplicity of independent unitary matrices. The key question is what
to do with the ensemble of reduced dimension covariance estimates.

\subsection{Notation and sample covariance}

We are given a $M \times N$ data matrix, $X$, the columns of which
comprise $N$ independent identically distributed realizations of a
random vector. For convenience we assume that the random vector is
zero-mean. We also will assume that the random vector is
circularly-symmetric complex. The sample covariance is
\begin{equation}
K = \frac{1}{N}XX^{*}.
\end{equation}
We are interested in the case where $N < M$. Consequently the sample
covariance is singular with rank equal to $N$.

\subsection{Dimensionality--reducing ensemble}

We introduce an ensemble of $L \times M$ random unitary matrices,
${ \Phi }$ where $L \leq N$ and $\Phi \Phi^* = I_L$, where $I_L$ is
the $L \times L$ identity matrix. The multiplication of the data
matrix by the unitary matrix results in a data matrix of reduced
dimension, $L \times N$, which in turn produces a statistically
meaningful sample covariance matrix provided that $L$ is
sufficiently small compared with $N$,
\begin{equation}
\frac{1}{N} \Phi X X^* \Phi^* = \Phi K \Phi^* .
\label{reduced_cov}
\end{equation}

We need to specify the distribution of the random unitary matrix. One
possibility would be to use a random permutation matrix, the effect of
which would be to discard all but $L$ of the $M$ components of the
data vectors. Instead we utilize the Haar measure (sometimes
called the ``isotropically random'' distribution \cite{marzetta}). A
fundamental property of the Haar distribution is its invariance to multiplication
of the random unitary matrix
by an unrelated unitary matrix. Specifically, let $p(\Phi)$ be the joint
probability density for the components of the unitary matrix, and let
$\Theta$ be any unrelated $M \times M$ unitary matrix (i.e., either
$\Theta$ is deterministic, or it is statistically independent of
$\Phi$). Then $\Phi$ has Haar measure if and only if for all unitary $\Theta$
\begin{equation}
p(\Phi \Theta) = p(\Phi) .
\label{isorandom}
\end{equation}
Compared with the random permutation matrix, the Haar measure is more
flexible as it permits linear constraints to be imposed.

\subsection{Two nonsingular covariance estimates}

The generation of the ensemble of reduced-dimension covariance
estimates (\ref{reduced_cov}) is well--motivated. It is less obvious what
to do with this ensemble. We have investigated two approaches: $\mathrm{cov}$
which yields directly a non--singular estimate for the $M \times M$ covariance
matrix, and $\mathrm{invcov}$ which yields directly an estimate for the inverse $M
\times M$ covariance matrix.

\subsubsection{$\mathrm{cov}$\,}

If we project the $L \times L$ covariance (\ref{reduced_cov}) out to a
$M \times M$ covariance using the same random unitary matrix, and then
take the expectation over the unitary ensemble, we obtain
the following:
\begin{equation}
\mathrm{cov}_L(K) = \E_\Phi \Big ( \Phi^* (\Phi K \Phi^*) \Phi\Big).
\end{equation}
This expectation can be evaluated in closed form (either by
evaluating fourth moments, or by using Schur polynomials as shown later):
\begin{equation}
\mathrm{cov}_L(K) = \frac{L}{(M^2-1)M} \Big[(M L
-1)K + (M-L) \mathrm{Tr} (K) I_{M}\Big].
\label{upstairs_mean}
\end{equation}
Thus the procedure $\mathrm{cov}$ is equivalent to diagonal loading for a
particular pair of loading parameters. The dimensionality parameter,
$L$, determines the amount of diagonal loading. It is reasonable
to re-scale the covariance expression (\ref{upstairs_mean}) by the factor $M/L$
because the dimensionality reduction yields shortened feature vectors
whose energy is typically $L / M$ times the energy of the original
feature vectors. Note that we use the term energy to denote the $\norm{x}_{2}^{2}$ of a vector $x$. If the covariance is scaled in this manner then the
trace of the sample covariance is preserved.

Although it is both interesting and surprising that $\mathrm{cov}$ is equivalent
to diagonal loading, we instead pursue an approach which is better
motivated and which promises more
compelling action.

\subsubsection{$\mathrm{invcov}$}

We first invert the $L \times L$ covariance (\ref{reduced_cov}) (which
is invertible with probability one), project out to $M \times M$ using
the same unitary matrix, and then take the expectation over the
unitary ensemble to obtain the following:
\begin{equation}
\mathrm{invcov}_L(K) = \E_\Phi \Big( \Phi^* ( \Phi K
\Phi^*)^{-1} \Phi\Big).
\label{ICOV}
\end{equation}
The estimate $\rm{invcov}$ (as well as $\rm{cov}$) preserves the eigenvectors. In other words, if we perform the eigenvector and eigenvalue
decomposition,
\begin{equation}
K = U D U^{*},
\label{K_eigen}
\end{equation}
where $D$ is the $M \times M$ diagonal matrix, whose diagonals are the
eigenvalues, ordered from largest to smallest, and $U$ is the $M
\times M$ unitary matrix of eigenvectors, then we prove (in Section \ref{derivation}) that 
\begin{equation}
\mathrm{invcov}_L(K)=U\mathrm{invcov}_L(D)U^{*}.
\end{equation}
Therefore it is enough to compute $\mathrm{invcov}_L(D)$. We also show that $\mathrm{invcov}_L(D)$ is a diagonal matrix. Moreover, we show that if $D=\mathrm{diag}(D_{N},0_{M-N})$ where $D_{N}=\mathrm{diag}(d_{1},\ldots,d_{N})$ is the matrix with the non--zero entries. The matrix $\mathrm{invcov}_{L}(D)$ is a diagonal matrix that can be decomposed as 
$$
\mathrm{invcov}_{L}(D)=\mathrm{diag}(\lambda_{1},\ldots,\lambda_{N},\mu I_{M-N}).
$$
In other words all the zero--eigenvalues are transformed to a non--zero constant $\mu$. In Section \ref{main_res} we prove an exact expression for the entries of $\mathrm{invcov}_{L}(D)$. More specifically we prove that  
$$
\mu=\mathrm{Tr}\Big[\E\big((X^{*}D_{N}X)^{-1}\big)\Big]
$$ 
where the average is taken over the ensemble of all $N\times L$ Gaussian random matrices $X$ with independent and complex entries with zero mean and unit variance. Proposition \ref{mainprop} (in Section \ref{main_res}) gives us an explicit formula for $\mu$. On the other hand, using Lemma \ref{lemma2} (in the same Section) we prove that
$$
\lambda_{k}=\frac{\partial}{\partial d_{k}}\int_{\Omega_{L,N}}{\mathrm{Tr}\log(\Phi^{*}D_N\Phi)\,d\phi}.
$$
where 
$$
\int_{\Omega_{L,N}}{\mathrm{Tr}\log(\Phi^{*}D_N\Phi)\,d\phi}
$$
can be explicitly computed using Theorem \ref{main} in Section \ref{main_res}. Therefore, given $D$ we obtain close form expressions for all the entries of the matrix $\mathrm{invcov}_{L}(D)$ for every $M,N$ and $L$.

\par In Section \ref{asymp} using Free Probability techniques we prove asymptotic formulas for the entries of $\mathrm{invcov}_{L}(D)$ for large values of $N$.

\par We focus the remainder of the paper on some potential applications of $\mathrm{invcov}$, the derivation of its fundamental properties, and how to
compute it. 

\section{Potential Applications of $\mathrm{invcov}$}

Typically neither the covariance matrix nor its inverse is of
direct interest. Rather some derived quantity is desired. Here we discuss
three potential applications where $\mathrm{invcov}$ arises in a natural way. 

\subsection{Design of a linear estimator from training data}

The problem is to design a minimum mean square linear estimator for a
$M_x \times 1$ random vector $x$ given an observation of a
$M_y \times 1$ random vector $y$. Exact statistics are not
available; instead we have to work with statistics that are estimated
from a set of training data. If the statistics were
available then the optimum estimator would be (assuming that the vectors
have zero--mean)
\begin{equation}
\hat{x}(y) = K_{xy}K_y^{-1} y,
\label{opt_est}
\end{equation}
where $K_{y}$ is the covariance matrix of vector $y$ and $K_{xy}$ is the cross--covariance matrix of vectors $x$ and $y$. In this case, the 
mean-square error is 
\begin{eqnarray}
\rm{MMSE} &=& \E \left( \left( \hat{x}(y) -x\right) \left( \hat{x}(y) -x\right)^*  \right) \nonumber \\
&=& K_x - K_{xy}K_y^{-1}K_{yx}.
\label{opt_mmse}
\end{eqnarray}
For the design of the estimator we have training data comprising $N$
independent 
joint realizations of $x$ and $y$: $X$ ($M_x
\times N$) and $Y$ ($M_y \times N$), where $N < M_y$.

We introduce an ensemble of $L \times M_y$ isotropically random
unitary matrices, $\Phi$, where $L \leq N$. We reduce the dimensionality
of the observed 
vector, $y \rightarrow \Phi y$, and the training
set, $Y \rightarrow \Phi Y$, and we estimate the relevant covariances
as follows,
\begin{equation}
K_{x,\Phi y} = \frac{1}{N} X Y^* \Phi^* ,
\label{covhat_x_Phiy}
\end{equation}
\begin{equation}
K_{\Phi y} = \frac{1}{N} \Phi Y Y^* \Phi^* .
\label{covhat_Phiy}
\end{equation}
We estimate $x$ given the reduced observation $\Phi y$ 
by treating the covariance estimates
(\ref{covhat_x_Phiy}) and (\ref{covhat_Phiy}) as if they were correct:
\begin{eqnarray}
\hat{x}(\Phi y) & = & K_{x,\Phi y}
K_{\Phi y}^{-1} \Phi y \nonumber \\ 
& = & X Y^* \Phi^* \left( \Phi Y Y^*
\Phi^*  \right)^{-1} \Phi y .
\label{hatx_Phiy}
\end{eqnarray}
The mean-square error of this estimator conditioned on the random
unitary matrix, $\Phi$,  is found by taking an expectation with respect
to the training data, $\{ 
X,Y \}$, the observation, $y$, (which is independent of the training
data), and the true value of the unknown vector, $x$:
\begin{equation}\label{mse_hatx_Phiy}
\E\left\{ \left[ \hat{x}(\Phi y)  -
x \right] \left[ \hat{x}(\Phi y) - x  \right]^*
\mid \Phi   \right\} =
\end{equation}
\begin{equation*}
\left[ K_x - K_{xy} \Phi^* \left( \Phi K_y \Phi^*  \right)^{-1}
\Phi K_{yx}  \right]\left[ 1 + \E \big( \mathrm{tr} \left(
(V^* V)^{-1}  \right) \big)  \right]
\end{equation*}
where $V$ is a $N \times L$ random matrix comprising independent CN(0,1) random
variables. We note the asymptotic result,
\begin{equation}
\E \Big( \mathrm{tr} \left(
(V^* V)^{-1}  \right) \Big) \stackrel{N,L \to
\infty}{\longrightarrow} \frac{L}{N-L} .
\end{equation}
The mean--square error (\ref{mse_hatx_Phiy}) is equal to the product of
two terms: the mean-square error which results from performing
estimation with a reduced observation vector and with exact statistics
available, and a penalty term which account for the fact that exact
statistics are not available. The first term typically decreases with
increasing dimensionality parameter, $L$, which the second term
increases with $L$.

Instead of performing the estimation using one value of the
dimensionality-reducing matrix, $\Phi$, one can average the estimator
(\ref{hatx_Phiy}) over the unitary ensemble:
\begin{eqnarray}
\hat{x} (y) & = & \E_\Phi \big( \hat{x}(\Phi
y) \big) \nonumber \\
& = & X Y^* \E_\Phi \Big( \Phi^* (\Phi Y Y^* \Phi^* )^{-1}
\Phi  \Big) y \nonumber \\
& = & X Y^* \cdot \mathrm{invcov}_L (Y Y^* ) \cdot y .
\label{hatx_ensemble}
\end{eqnarray}
Jensen's inequality implies that the ensemble-averaged estimator
(\ref{hatx_ensemble}) has better performance than the estimator
(\ref{hatx_Phiy}) that
is based on a single realization of $\Phi$,
\begin{equation*}
\E \Big(  \big[ \E_\Phi (\hat{x} (\Phi y)) - x   \big] \cdot \big[ \E_\Phi (\hat{x} (\Phi y))-x \big]^*  \Big) \leq 
\end{equation*}
\begin{equation}
\E_\Phi \Big( \E \big( \left[ \hat{x} (\Phi y) -
x   \right] \left[ \hat{x} (\Phi y) - x   \right]^*  \big) \Big).
\label{jensen_mse}
\end{equation} 

\subsection{Supervised learning: Design of a quadratic classifier
from training data}

The problem is to design a quadratic classifier from labeled training
data. Given an observation of a $M \times 1$ zero-mean
complex Gaussian random vector, the classifier has to choose one of
two hypotheses. Under hypothesis $H_j$, $j=0,1$, the observation is
distributed as $\mathrm{CN}(0,K_j)$, $j=0,1$. If the two covariance
matrices were known the optimum classifier is a ``likelihood ratio test" 
\cite{Duda_Hart},
\begin{equation}
- x^* \left( K_1^{-1} - K_0^{-1} \right) x
  \begin{array}{c} H_1 \\ > \\ < \\ H_0 \end{array} \gamma ,
\label{lrt}
\end{equation}
where $\gamma$ is a threshold. Instead the covariances have to be
estimated from two $M \times N$ matrices of labeled training data,
$X_j$,  $j=0,1$, each of which comprises $N<M$ independent
observations of the random vector under their respective hypotheses.

We introduce an ensemble of $L \times M$ random unitary matrices,
$\Phi$, where $L \leq N$. For a given $\Phi$ we reduce the dimension
of both sets of training data and then estimate the reduced covariance
matrices,
\begin{equation}
K_j = \frac{1}{N} \Phi X_j X_j^* \Phi^* , \: j=0,1 .
\label{cov_training}
\end{equation}
For any $\Phi$ we could implement a likelihood ratio test based on the
estimated reduced covariances (\ref{cov_training}) and the reduced
observation, $\Phi x$. Alternatively we could base the hypothesis test
on the expectation of the log--likelihood ratio with respect to the
unitary ensemble,
\begin{equation*}
- x^* \E_\Phi \Big( \Phi^* \big( (\Phi X_1 X_1^* \Phi^* )^{-1} -
  (\Phi X_0 X_0^* \Phi^* )^{-1} \big) \Phi  \Big)  x = 
\end{equation*}
\begin{equation}
- x^* \left( \mathrm{invcov}_L(X_1 X_1^*) - \mathrm{invcov}_L(X_0
  X_0^*) \right) x \begin{array}{c} H_1 \\ > \\ < \\ H_0 \end{array} \gamma.
\label{lrt_icov}
\end{equation}
This classifer is of the ``naive Bayes'' type \cite{naive_bayes}, in
which statistical 
dependencies (in this case the individual likelihood ratios are not
statistically dependent) are ignored in order to simplify the
construction of the classifier. 

\subsection{Capon MVDR spectral estimator}

The Capon MVDR (minimum variance distortionless response) spectral
estimator estimates power as a function of angle-of-arrival given $N$
independent realizations of a $M$-dimensional measurement vector from
an array of sensors \cite{mvdr}. Let $X$ be the $M \times N$ vector of
measurements, $M < N$, and let the ``steering vector'', $a$, be the
$M$ dimensional unit vector which describes the wavefront at the
array. The conventional power estimate, as a function of the steering
vector, is
\begin{equation}
P_{\mathrm{conv}} = a^* K a ,
\label{p_conv}
\end{equation}
where $K$ is the sample covariance matrix. The Capon MVDR power
estimate is
\begin{equation}
P_{\mathrm{Capon}} = \frac{1}{a^* K^{-1} a} .
\label{p_capon}
\end{equation}
A justification for the Capon estimator is the following: one
considers the estimated covariance matrix 
to be the sum of two terms, the first corresponding to power arriving
from the direction that is specified by the steering vector, and the
second corresponding to power arriving from all other directions,
\begin{equation}
K = P \cdot a a^* + K_{\mathrm{other}} .
\label{K_decomp}
\end{equation}
It can be shown that the Capon power estimate (\ref{p_capon}) is
equal to the largest value of power $P$ such that, in the
decomposition (\ref{K_decomp}), $K_{\mathrm{other}}$ is nonnegative
definite \cite{marzetta_capon}. In other words the decomposition
(\ref{K_decomp}) is nonunique, and the Capon power estimate is an
upper bound on the possible value that the power can take.

We deal with the singularity of the covariance matrix by introducing
an ensemble of $L \times M$ unitary matrices, $\Phi$. Since we are
looking for power that arrives from a particular direction we
constrain the unitary matrices to preserve the energy of the steering
vector, i.e., $a^* \Phi^* \Phi a = a^* a = 1$. This is readily done
through a Householder unitary matrix, 
$Q$, such that
\begin{equation}
Q = \left[ \begin{array}{c} a^* \\ A_\perp \end{array}  \right] ,
\label{householder}
\end{equation}
where $A_\perp$ is a $M-1 \times M$ unitary matrix whose rows are
orthogonal to $a$. We represent the ensemble $\Phi$ as follows:
\begin{equation}
\Phi = \left[ \begin{array}{c} a^* \\ \Theta A_\perp \end{array}  \right],
\label{phi_constrained}
\end{equation}
where $\Theta$ is a $L-1 \times M-1$ isotropically random unitary
matrix. We now use the constrained unitary matrix $\Phi$ to reduce the
dimensionality of the sample covariance matrix and the steering
vector, we compute the Capon power estimate from the reduced
quantities, and finally we average 
the power with respect to the unitary ensemble \cite{marzetta_simon}:
\begin{eqnarray}
\hat{P} & = & \E_\Phi \Bigg[ \frac{1}{a^* \Phi^* \left( \Phi
K \Phi^*  \right)^{-1} \Phi a} \Bigg] \\
& = & a^* K a - a^* K A_\perp^* \,\mathrm{invcov}_{L-1} ( A_\perp
K A_\perp^* ) A_\perp K a \nonumber
\label{p_capon_reduced}
\end{eqnarray}
where
\begin{equation}
\mathrm{invcov}_{L-1} ( A_\perp K A_\perp^* ) = \E_\Theta
\Big(  \Theta^* \left(  \Theta A_\perp K A_\perp^* \Theta^*
\right)^{-1} \Theta  \Big).
\label{icov_theta}
\end{equation}

\subsection{Distantly related research}

Our approach to handling covariance singularity is based on an
ensemble of dimensionality--reducing random unitary matrices. Here we
mention some other lines of research which also involve random
dimensionality reduction.

\subsubsection{Johnson--Lindenstauss Lemma}

In qualitative terms, the Johnson--Lindenstrauss Lemma
\cite{Johnson_Lindenstrauss} has the following implication: the angle
between two vectors of high 
dimension tends to be preserved accurately when the vectors are
shortened through multiplication by a random unitary
dimensionality--reducing matrix.

\subsubsection{Compressive Sampling or Sensing}

Compressive sampling or sensing permits the recovery of a
sparsely-sampled data vector (for example, obtained by multiplying the
original vector by a random dimensionality--reducing matrix), provided
the original data vector can be linearly transformed to a domain in
which it has sparse support \cite{Compressive_sampling}. Compressive
sampling utilizes only one dimensionality--reducing matrix. In contrast
our approach to handling covariance singularity utilizes an
\emph{ensemble} of random dimensionality--reducing matrices.

\section{Derivation of Some Basic Properties of $\mathrm{invcov}$}\label{derivation}

In this Section we state and prove two basic and fundamental properties of
$\mathrm{invcov}_L(K)$. We perform the eigenvector and eigenvalue
decomposition,
\begin{equation}
K=UDU^{*},
\end{equation}
where $D$ is the $M \times M$ diagonal matrix, whose diagonals are the
eigenvalues, ordered from largest to smallest, and $U$ is the $M
\times M$ unitary matrix of eigenvectors.

\subsection{Eigenvectors of sample covariance are preserved}

We substitute the eigenvalue decomposition (\ref{K_eigen}) into the expression
(\ref{ICOV}) for $\mathrm{invcov}_L(K)$ to obtain the following:
\begin{eqnarray}
\mathrm{invcov}_L(K)
& = & \E_\Phi \Big( \Phi^* \left( \Phi K
\Phi^* \right)^{-1} \Phi \Big) \nonumber \\
& = & \E_\Phi \Big( U (\Phi U)^* \left( \Phi U D U^*
\Phi^* \right)^{-1} (\Phi U) U^* \Big) \nonumber \\
& = & U \E_\Phi \Big( \Phi^* \left( \Phi D
\Phi^* \right)^{-1} \Phi \Big) U^* \nonumber \\
& = & U \mathrm{invcov}_L(D) U^* 
\label{ICOV_2}
\end{eqnarray}
where we have used the fundamental definition of the isotropic
distribution (\ref{isorandom}), i.e. that the product $\Phi U$ has the
same distribution as $\Phi$. We intend to show that
$\mathrm{invcov}_L(D)$ is itself diagonal. We utilize the fact that a
matrix $A$ is diagonal if and only if, for all diagonal unitary
matrices, $\Omega$, $\Omega A \Omega^* = A$. Let $\Omega$ be a
diagonal unitary matrix, we have
\begin{eqnarray}
& & \Omega\,\,\mathrm{invcov}_L(D) \,\,\Omega^* = \nonumber \\
& = & \E_\Phi \Big( (\Phi \Omega^*)^* \big( (\Phi \Omega^*)
\Omega D \Omega^* (\Phi \Omega^*)^* \big)^{-1} (\Phi \Omega^*)
\Big) \nonumber \\
& = & \E_\Phi \Big( \Phi^* \big( \Phi
 D \Phi^* \big)^{-1} \Phi
\Big) \nonumber \\
& = &  \mathrm{invcov}_L(D)
\label{proof_eig}
\end{eqnarray}
where we used the fact that $\Phi \Omega^*$ has the same distribution
as $\Phi$, and that $\Omega D \Omega^* = D$. Therefore we have
established that the final expression in (\ref{ICOV_2}) is the
eigenvector/eigenvalue decomposition of $\mathrm{invcov}_L(K)$, for
which the
eigenvector matrix is $U$ and the diagonal matrix of eigenvalues is
$\mathrm{invcov}_L(D)$. Hence, we need only consider applying $\mathrm{invcov}$ to diagonal matrices.

\subsection{The zero-eigenvalues of the sample covariance are
converted to equal positive values}

When the rank of the covariance matrix is equal to $N < M$, the
eigenvalue matrix of $K$ has the form
\begin{equation}
D = \left[ \begin{array}{cc} D_N & 0 \\ 0 & 0_{M-N} \end{array}   \right].
\label{D_N}
\end{equation}
We want to establish that the last $M-N$ eigenvalues of
$\mathrm{invcov}_L(K)$ are equal. To that end we introduce a unitary
matrix, $\Xi$,
\begin{equation}
\Xi = \left[ \begin{array}{cc} I_N & 0 \\ 0 & P_{M-N} \end{array}
\right],
\label{permutation}
\end{equation}
 where $P_{M-N}$ is an arbitrary $M-N \times M-N$ permutation
 matrix. We now pre-- and post--multiply $\mathrm{invcov}_L(D)$ by $\Xi$
 and $\Xi^*$ respectively: it will be shown that this does not change
 the diagonal matrix, so consequently the last $M-N$ eigenvalues are
 equal. We have
\begin{eqnarray}
\Xi \,\,\mathrm{invcov}_L(D)\,\, \Xi^* & = & \Xi \,\,\mathrm{E}_\Phi \Big( \Phi^* (\Phi D \Phi^*)^{-1} \Phi
\Big) \,\, \Xi^* \nonumber \\
& = & \E_\Phi \Big( \Phi^* (\Phi \Xi D \Xi^* \Phi^*)^{-1}
\Phi \Big) \nonumber \\
& = & \E_\Phi \Big( \Phi^* (\Phi D \Phi^*)^{-1}
\Phi \Big) \nonumber \\
& = & \mathrm{invcov}_L(D),
\label{proof_eig_2}
\end{eqnarray}
where we used the fact that $\Phi \Xi^*$ has the same distribution as
$\Phi$, and that $\Xi D \Xi^* = D$.

\section{Functional Equation}
\vspace{0.4cm}
\par In this Section we will prove a functional equation for the inverse covariance estimate $\mathrm{invcov}_{L}(K)$.

\par Let $K$ be an $M\times M$ sample covariance matrix $K$ of rank $N$. Since $K$ is positive definite there exists $U$ an $M\times M$ unitary and $D$ an $M\times M$ diagonal matrix of rank $N$ such that $K=UDU^{*}$. Fix $L\leq N$. We would like to compute, 
\begin{equation}
\mathrm{invcov}_{L}(K)  = \E(\Phi^{*}(\Phi K\Phi^{*})^{-1}\Phi)
\end{equation}
where $\Phi$ is an $L\times M$ unitary matrix and the average is taken with respect to the isotropic measure. Let $Z$ be an $L\times M$ Gaussian random matrix with complex, independent and identically distributed entries with zero mean and variance 1. It is a well known result in random matrix theory (see \cite{Ver}) that we can decompose $Z=C\Phi$ where $C$  is an $L\times L$ positive definite and invertible matrix (with probability one). Hence, $Z^{*}(ZKZ^{*})^{-1}Z=\Phi^{*}(\Phi K\Phi^{*})^{-1}\Phi$. Therefore,
\begin{equation}
\mathrm{invcov}_{L}(K)=\E\Big(Z^{*}(ZKZ^{*})^{-1}Z\Big).
\end{equation}
Moreover, as shown in the previous Section  
\begin{equation}
\mathrm{invcov}_{L}(K)=U\E\Big(Z^{*}(ZDZ^{*})^{-1}Z\Big)U^{*}.
\end{equation}
\noindent Therefore it is enough to compute $\E(Z^{*}(ZDZ^{*})^{-1}Z)$. Decompose $Z$ as $Z=[X,Y]$ where $X$ is $L\times N$ and $Y$ is $L\times (M-N)$. Now performing the block matrix multiplications and taking the expectation we obtain that $\mathrm{invcov}_{L}(D)$ is equal to 
\begin{equation}\label{eq_fund}
\Bigg[\begin{array}{ll}
\E(X^{*}(XD_{N}X^{*})^{-1}X) & 0 \\ 
0 & \E(Y^{*}(XD_{N}X^{*})^{-1}Y)\end{array}\Bigg]
\end{equation}
\noindent where $D=\mathrm{diag}(D_{N},0_{M-N})$ and $D_{N}$ an $N\times N$ diagonal matrix of full rank.

\vspace{0.3cm}
\noindent Let us first focus on the $N\times N$ matrix $\E(X^{*}(XD_{N}X^{*})^{-1}X)$, denote this matrix by  
\begin{equation}
\Lambda_{L}(D_{N}):=\E(X^{*}(XD_{N}X^{*})^{-1}X).
\end{equation}
Let $W$ be the matrix $W=XD_{N}^{1/2}$. Then
\begin{eqnarray}
\Lambda_{L}(D_{N}) & = & E(X^{*}(XD_{N}X^{*})^{-1}X)\\
& = & D_{N}^{-1/2}\E(W^{*}(WW^{*})^{-1}W)D_{N}^{-1/2}\nonumber
\end{eqnarray}
\noindent and
\begin{eqnarray*}
\E(W^{*}(WW^{*})^{-1}W) & = & \lim_{\gamma\to 0}{\,\E\Big[W^{*}(\gamma I_{L}+WW^{*})^{-1}W\Big]}\\
& = & I_{N}-\lim_{\gamma\to 0}{\E\Big[\big(I_{N}+\frac{1}{\gamma}W^{*}W\big)^{-1}\Big]}.
\end{eqnarray*}
\noindent Therefore it is enough to compute 
$$
\lim_{\gamma\to 0}{\,\,\E\Big[\Big(I_{N}+\frac{1}{\gamma}W^{*}W\Big)^{-1}\Big]}
$$
\noindent which is equal to 
$$
\lim_{\gamma\to 0}{\,\,D_{N}^{-1/2}\,\E\Big[\Big(D_{N}^{-1}+\frac{1}{\gamma}X^{*}X\Big)^{-1}\Big]D_{N}^{-1/2}}.
$$
\noindent Let us decompose $X^{*}X=\Omega D_{x}\Omega^{*}$ where $\Omega$ is an $N\times N$ unitary matrix and $D_{x}$ is an $N\times N$ diagonal matrix of rank $L$. Then $D_{x}=\mathrm{diag}(D_{0},0_{N-L})$. It is a straightforward calculation to see that 
$$
\,\,D_{N}^{-1/2}\,\E\Big[\Big(D_{N}^{-1}+\frac{1}{\gamma}X^{*}X\Big)^{-1}\Big]D_{N}^{-1/2}
$$
\noindent it is equal to
$$
D_{N}^{-1/2}\,\E\Big[\Omega\Big(\Omega^{*}D_{N}^{-1}\Omega+\frac{1}{\gamma}D_{x}\Big)^{-1}\Omega^{*}\Big]D_{N}^{-1/2}.
$$
\noindent Doing the block matrix decomposition 
$$
\Omega^{*}D_{N}^{-1}\Omega=\Bigg[
\begin{array}{ll}
A_{11} & A_{12} \\ 
A_{21} & A_{22}\end{array}\Bigg]
$$
\noindent where $A_{11}$ is $L\times L$ and $A_{22}$ is $(N-L)\times (N-L)$ we see that 
\begin{eqnarray*}
\Big(\Omega^{*}D_{N}^{-1}\Omega+\frac{1}{\gamma}D_{x}\Big)^{-1} & = & \Bigg[
\begin{array}{ll}
A_{11}+\frac{1}{\gamma}D_{0} & A_{12} \\ 
A_{21} & A_{22}\end{array}\Bigg]^{-1}\\
& = & \Bigg[
\begin{array}{ll}
X & Y \\ 
Z & W\end{array}\Bigg]
\end{eqnarray*}
\noindent where $X=\big(A_{11}+\frac{1}{\gamma}D_{0}-A_{12}A_{22}^{-1}A_{21}\big)^{-1}$, $Y=-XA_{12}A_{22}^{-1}$, $Z=-A_{22}^{-1}A_{21}X$ and $W=A_{22}^{-1}+A_{22}^{-1}A_{21}XA_{12}A_{22}^{-1}$. Since $\lim_{\gamma\to 0}{X}=0$ we see that 
$$\lim_{\gamma\to 0}{\Big(\Omega^{*}D_{N}^{-1}\Omega+\frac{1}{\gamma}D_{x}\Big)^{-1}}=\Bigg[
\begin{array}{ll}
0 & 0 \\ 
0 & A_{22}^{-1}\end{array}\Bigg].$$
\noindent Putting all the pieces together we obtain that
\begin{equation}\label{EQ1}
\Lambda_{L}(D_{N})=D_{N}^{-1}-D_{N}^{-1}\cdot\mathbb{E}\Bigg(\Omega\Bigg[\begin{array}{ll}
0 & 0 \\
0 & A_{22}^{-1}\end{array}\Bigg]\Omega^{*}\Bigg)\cdot D_{N}^{-1}
\end{equation}
\noindent where $\Omega$ is an isotropically distributed $N\times N$ unitary matrix and
$$\Omega^{*}D_{N}^{-1}\Omega=\Bigg[\begin{array}{ll}
A_{11} & A_{12} \\
A_{21} & A_{22}\end{array}\Bigg].$$

\vspace{0.3cm}

\noindent Let us decompose the unitary matrix $\Omega$ as $\Omega=[\Omega_{1}\,\,\, \Omega_{2}]$, where $\Omega_{1}$ is $N\times L$ and $\Omega_{2}$ is $N\times (N-L)$ matrix. Then $\Omega_{1}^{*}\Omega_{1}=I_{L}$ and $\Omega_{2}^{*}\Omega_{2}=I_{N-L}$ isotropically distributed unitaries. It is an easy calculation to see that
\begin{eqnarray}\label{EQ2}
\mathbb{E}\Bigg(\Omega\Bigg[\begin{array}{ll}
0 & 0 \\
0 & A_{22}^{-1}\end{array}\Bigg]\Omega^{*}\Bigg) & = & \mathbb{E}\Big( \Omega_{2}\big(\Omega_{2}^{*}D_{N}^{-1}\Omega_{2}\big)^{-1}\Omega_{2}^{*} \Big) \nonumber \\ 
& = & \Lambda_{N-L}(D_N^{-1}).
\end{eqnarray}

\vspace{0.3cm}
\noindent Therefore, using equation (\ref{EQ1}) and equation (\ref{EQ2}) we found the following functional equation
\vspace{0.2cm}
\begin{equation}\label{funct_eq}
D_{N}\Lambda_{L}(D_{N})+D_{N}^{-1}\Lambda_{N-L}(D_{N}^{-1})=I_{N}.
\end{equation}

\vspace{0.5cm}
\begin{remark}
Here we list a few results on $\Lambda_{L}(D_N)$.
\begin{enumerate}
	\item If $N=L$ then $D_{N}\Lambda_{N}(D_{N})=I_N$ and therefore $\Lambda_{N}(D_{N})=D_{N}^{-1}$.
	 
	\item It is not difficult to see, and well known result on random matrices, see \cite{Ver}, that $\Lambda_{L}(I_N)=\frac{L}{N}I_N$. Therefore, 
	\begin{equation}\label{eq_d0}
	\Lambda_{L}(\alpha I_N)=\frac{L}{\alpha N} I_N.  
	\end{equation}
	Hence,
	$$
	\alpha I_N\Lambda_L(\alpha I_N)+\alpha^{-1}I_{N}\Lambda_{N-L}(\alpha^{-1} I_N)=I_N
	$$
	which agrees with equation (\ref{funct_eq}).

	\item $\mathrm{Tr}(D_{N}\Lambda_{L}(D_{N}))=\mathrm{Tr}(\E(D_{N}X^{*}(XD_{N}X^{*})^{-1}X))=L$ where in the last equality we used the trace property.

\end{enumerate}
\end{remark}

\vspace{0.3cm}
\par As we saw in Equation (\ref{eq_fund}) the other important term in $\mathrm{invcov}_{L}(D)$ is 
$$
\E(Y^{*}(XD_{N}X^{*})^{-1}Y).
$$
Let us define $\mu>0$ as 
\begin{equation}\label{eqmu}
\mu:=\mathrm{Tr}\Big[\E((XD_{N}X^{*})^{-1})\Big].
\end{equation}
Since $X$ and $Y$ are Gaussian independent random matrices it is clear that 
\begin{eqnarray*}
\E(Y^{*}(XD_{N}X^{*})^{-1}Y) & = & \mathrm{Tr}\Big[\E((XD_{N}X^{*})^{-1})\Big]I_{M-N}\\
& = & \mu I_{M-N},
\end{eqnarray*}
where $I_{M-N}$ is the identity matrix of dimension $M-N$. Putting all the pieces together we see that the estimate $\mathrm{invcov}_{L}(D)$ is equal to 
\begin{equation}
\mathrm{invcov}_{L}(D)=\mathrm{diag}(\Lambda_{L}(D_{N}),\mu I_{M-N}).
\end{equation}

\section{fcov exact formula}\label{main_res}

\par In this Section we will prove an exact and close form expression for the entries of $\mathrm{invcov}_{L}(D)$. We will treat separately the entries of $\Lambda_{L}(D_{N})$ and the constant term $\mu$. As a matter of fact the analysis developed in this Section will allow us to obtain close form expressions for more general averages. 

\par Recall that we say a matrix $A$ is said to be normal if it commutes with its conjugate transpose $AA^{*}=A^{*}A$. Given a normal matrix $A$ and $f:\C\to\C$ a continuous function we can always define $f(A)$ using functional calculus. Being more precise we know by the spectral Theorem that exist $U$ unitary and $D=\mathrm{diag}(d_1,\ldots,d_N)$ such that 
$$
A=UDU^{*}.
$$ 
We then define 
$$
f(A)=UD_{f}U^{*}
$$
where $D_{f}=\mathrm{diag}(f(d_1),\ldots,f(d_N))$. In particular, let $D$ be as before and let $f$ be a continuous function, we will obtain an exact expression for
\begin{equation}
\mathrm{fcov}_{L}(D):=\E\Big(\Phi^{*}f(\Phi D\Phi^{*})\Phi\Big)
\end{equation}
where $\Phi$ is an $L\times N$ unitary isotropically random. Note that our covariance estimate $\mathrm{invcov}_{L}(D)$ is a particular case of the last expression when $f(x)=x^{-1}$.

\vspace{0.3cm}
\noindent Let $\Omega_{L,N}=\{\Phi\in\C^{N\times L}\,\,:\,\, \Phi^{*}\Phi=I_{L}\}$ be the Stiefel manifold with the isotropic measure $d\phi$. By equation (18) in \cite{Fyo} we know that 
\begin{equation}\label{intschur}
\int_{\Omega_{L,N}}{s_{\lambda}(\Phi^{*}D_N\Phi)d\phi}=\frac{s_{\lambda}(D_N)s_{\lambda}(I_{L})}{s_{\lambda}(I_{N})} 
\end{equation}

\noindent where $s_{\lambda}$ is the Schur polynomial associated with the partition $\lambda$. The latter are explicitly defined for any $N\times N$ matrix $A$ in terms of the eigenvalues $a_{1},\ldots,a_{N}$ as 
\begin{equation}
s_{\lambda}(A)=s_{\lambda}(a_{1},\ldots,a_{N})=\frac{\det(a_{i}^{N+\lambda_{j}-j})_{i,j=1}^{N}}{\det(a_{i}^{N-j})_{i,j=1}^{N}},
\end{equation}
\noindent with $\lambda$ being a partition, i.e. a non--increasing sequence of non--negative integers $\lambda_{j}$. For an introduction to the theory of symmetric functions and properties of the Schur polynomials see \cite{Mac} and \cite{Muir}.

\vspace{0.3cm}
\noindent Denote by $(n-k,1^{k})$ the partition $(n-k,1,1,\ldots,1)$ with $k$ ones. One of the properties of the Schur polynomials is that 
\begin{equation}\label{traceschur}
\mathrm{Tr}(A^{n})=\sum_{k=0}^{N-1}{(-1)^{k}s_{(n-k,1^{k})}(A)} 
\end{equation}

\noindent Using equation (\ref{intschur}) and (\ref{traceschur}) we see that 
$$
\int_{\Omega_{L,N}}{\mathrm{Tr}\Big((\Phi^{*}D_N\Phi)^{n}\Big)\,d\phi}
$$
is equal to 
\begin{equation}\label{x^n}
\sum_{k=0}^{L-1}{(-1)^{k}\,\,\frac{s_{(n-k,1^{k})}(D_N)s_{(n-k,1^{k})}(I_{L})}{s_{(n-k,1^{k})}(I_{N})}}. 
\end{equation}

\noindent The constant $s_{(n-k,1^{k})}(I_{p})=\frac{(n+p-k-1)!}{k!(p-k-1)!(n-k-1)!n}$ see \cite{Mac}. Therefore, 
\begin{equation}\label{const}
\frac{s_{(n-k,1^{k})}(I_{L})}{s_{(n-k,1^{k})}(I_{N})}=\frac{(n+L-(k+1))!}{(n+N-(k+1))!}\cdot\frac{(N-(k+1))!}{(L-(k+1))!}.
\end{equation}

\vspace{0.3cm}

\noindent For each $p\geq 0$ consider the operator $I^{(p)}$ defined in $x^{n}$ by $I^{(p)}(x^{n})=\frac{x^{n+p}}{(n+1)\ldots(n+p)}$. This extends linearly and continuously to a well defined linear operator $I^{(p)}:C[0,r]\to C[0,r]$ where $C[0,r]$ are the continuous functions in the interval $[0,r]$. Now we are ready to state the main Theorem of this Section.

\vspace{0.3cm}
 
\begin{theorem}\label{main}
Let $D_N$ be an $N\times N$ diagonal matrix of rank $N$. For any continuous (complex or real valued) function $f\in C[d_{\mathrm{min}},d_{\mathrm{max}}]$ 
\begin{equation}\label{f_eq}
\int_{\Omega_{L,N}}{\mathrm{Tr}\Big(f(\Phi^{*}D_N\Phi)\Big)\,d\phi}
\end{equation}
is equal to 
$$
\sum_{k=0}^{L-1}{\frac{(N-(k+1))!}{(L-(k+1))!}\cdot\frac{\det(G_{k})}{\det(\Delta(D_N))}} 
$$
where $\Delta(D_N)$ is the Vandermonde matrix associated to $D_N$ and $G_{k}$ is the matrix defined by replacing the $(k+1)$ row of the Vandermonde matrix $\Delta(D_N)$, $\{d_{i}^{N-(k+1)}\}_{i=1}^{n}$, by the row 
$$
\Big\{I^{(N-L)}(x^{(L-(k+1))}f(x))|_{x=d_i}\Big\}_{i=1}^{N}.
$$
\end{theorem}

\begin{proof}
By linearity and continuity (polynomials are dense in the set of continuous functions) it is enough to prove (\ref{f_eq}) in the case $f(x)=x^{n}$. By (\ref{x^n}) and (\ref{const}) we know that 
$$
\int_{\Omega_{L,N}}{\mathrm{Tr}\Big((\Phi^{*}D_N\Phi)^{n}\Big)\,d\phi}
$$
is equal to 
\begin{equation}
\sum_{k=0}^{L-1}{(-1)^{k}\,\,c_{k}^{(N,L)}\cdot s_{(n-k,1^{k})}(D_N)}
\end{equation}
where 
$$
c_{k}^{(N,L)}:=\frac{(N-(k+1))!}{(L-(k+1))!}\cdot\frac{(n+L-(k+1))!}{(n+N-(k+1))!}.
$$ 
\noindent By definition of the Schur polynomials (see \cite{Mac}) 
$$
s_{(n-k,1^{k})}(D_N)=\frac{\det(S_{k})}{\det(\Delta(D_N))}
$$ 
where the $i^{th}$--column of the matrix $S_{k}$ is 
$$\begin{pmatrix}
d_{i}^{N-1+n-k} \\ 
d_{i}^{N-2+1} \\
d_{i}^{N-3+1} \\
\vdots \\ 
d_{i}^{N-(k+1)+1} \\ 
d_{i}^{N-(k+2)}\\
\vdots \\
d_{i}^{N-(N-1)}\\
1\\
\end{pmatrix}.$$

\vspace{0.4cm}
\noindent Doing $k$ row transpositions on the rows of the matrix $S_{k}$ we obtain a new matrix $\tilde{S}_{k}$ where the $i^{th}$--column of this new matrix is 

$$
\begin{pmatrix}
d_{i}^{N-1} \\ 
d_{i}^{N-2} \\
d_{i}^{N-3} \\
\vdots \\ 
d_{i}^{N-k}\\
d_{i}^{N+n-(k+1)} \\ 
d_{i}^{N-(k+2)}\\
\vdots \\
d_{i}^{N-(N-1)}\\
1\\
\end{pmatrix}.
$$
Note that this matrix $\tilde{S}_{k}$ is identical to $\Delta(D_N)$ except of the $(k+1)$ row which was replaced by the row $\{d_{i}^{N+n-(k+1)}\}_{i=1}^{N}$ instead of $\{d_{i}^{N-(k+1)}\}_{i=1}^{N}$ as in $\Delta(D_{N})$. Therefore,
$$
\int_{\Omega_{L,N}}{\mathrm{Tr}\Big((\Phi^{*}D_N\Phi)^{n}\Big)\,d\phi}
$$
is equal to 
\begin{equation}
\sum_{k=0}^{L-1}{\,c_{k}^{(N,L)}\cdot \frac{\mathrm{det}(\tilde{S}_k)}{\mathrm{det}(\Delta(D_N))}}.
\end{equation}

\noindent Now and using the fact that 
$$
I^{(N-L)}(x^{n+L-(k+1)})\r{x=d_{i}}=\frac{(n+L-(k+1))!}{(n+N-(k+1))!} d_{i}^{N+n-(k+1)}
$$ 
we obtain that 
$$
G_{k}=\frac{(n+L-(k+1))!}{(n+N-(k+1))!}\cdot \tilde{S}_k.
$$ 
Putting all the pieces together we obtain that
$$
\int_{\Omega_{L,N}}{\mathrm{Tr}((\Phi^{*}D_N\Phi)^{n})\,d\phi}
$$
is equal to
$$
\sum_{k=0}^{L-1}{\frac{(N-(k+1))!}{(L-(k+1))!}\cdot\frac{\det(G_{k})}{\det(\Delta(D_N))}}. 
$$
\end{proof}

\vspace{0.3cm}
\noindent The next Proposition gives us, in particular, an exact and close form expression for the term $\mu$ in equation (\ref{eqmu}) discussed previously.

\vspace{0.3cm}
\begin{prop}\label{mainprop}
Let $D_N$ be an $N\times N$ diagonal matrix of rank $N$. Let $1\leq L<N$ and $X$ be an $N\times L$ Gaussian random matrix with independent and identically distributed entries with zero mean and variance $1$. Then
\begin{equation}
\int_{\Omega_{L,N}}{\mathrm{Tr}((\Phi^{*}D_N\Phi)^{-1})\,d\phi}=(N-L)\cdot\frac{\det(G)}{\det(\Delta(D_N))}
\end{equation}
and 
\begin{equation}
\mu:=\E\Big[\mathrm{Tr}\big((X^{*}D_N X)^{-1}\big)\Big]=\frac{\det(G)}{\det(\Delta(D_N))}
\end{equation}
where $G$ is the matrix constructed by replacing the $L^{th}$ row of the Vandermonde matrix by the row 
$$\Big(d_{1}^{N-(L+1)}\log(d_1),\ldots,d_{N}^{N-(L+1)}\log(d_{N})\Big).$$
\end{prop}

\begin{proof} 
Let $1\leq L<N$ and let us consider $f(x)=x^{-1}$ then for each $0\leq k\leq L-2$ we see that 
$$
I^{(N-L)}(x^{L-(k+1)} x^{-1})\r{x=d_{i}}=\frac{d_{i}^{N-(k+2)}}{(L-k-1)\ldots (N-k-2)}
$$ 
\noindent since this is a multiple of the $(k+2)^{th}$ row of the Vandermonde matrix $\Delta(D_N)$ we see that $\det(G_{k})=0$ for all $0\leq k\leq L-2$.\\ 

\noindent It is not difficult to see that for $k=L-1$  
$$I^{(N-L)}(x^{-1})\r{x=d_i}=\frac{d_{i}^{N-(L+1)}\log(d_{i})}{(N-(L+1))!}.$$ 

\noindent Therefore using Theorem \ref{main} we see that
\begin{equation}
\int_{\Omega_{L,N}}{\mathrm{Tr}((\Phi^{*}D_N\Phi)^{-1})\,d\phi}=(N-L)\cdot\frac{\det(G)}{\det(\Delta(D_N))}
\end{equation}

\noindent where $G$ is constructed by replacing the $L^{th}$ row of the Vandermonde matrix by the row 
$$\Big(d_{1}^{N-(L+1)}\log(d_1),\ldots,d_{N}^{N-(L+1)}\log(d_{N})\Big).$$

\vspace{0.3cm}
\noindent Now we will prove the second part of the Proposition. Given $X$ an $N\times L$ random matrix as in the hypothesis we can decompose $X$ as 
$$
X=\Phi C
$$
where $\Phi$ is an isotropically random $N\times L$ unitary and $C$ is a positive definite $L\times L$ independent from $\Phi$ (see Section 2.1.5 from \cite{Ver}). The matrix $C$ is invertible with probability 1. Hence,
$$
(X^{*}D_{N}X)^{-1}=(C\Phi^{*}D_{N}\Phi C)^{-1}=C^{-1}(\Phi^{*}D_{N}\Phi)^{-1}C^{-1}.
$$
Therefore, 
\begin{eqnarray*}
\E\Big[\mathrm{Tr}\big((X^{*}D_N X)^{-1}\big)\Big] & = & \E\Big[\mathrm{Tr}\big( C^{-1}(\Phi^{*}D_{N}\Phi)^{-1}C^{-1} \big)\Big]\\
& = & \E\Big[\mathrm{Tr}\big( C^{-2}(\Phi^{*}D_{N}\Phi)^{-1} \big)\Big]\\
\end{eqnarray*}
where in the second equality we used the trace property. Recall from random matrix theory that if $A,B$ are independent $n\times n$ random matrices then 
$$
\E\Big[\mathrm{Tr}(AB)\Big]=\frac{1}{n}\E\Big[\mathrm{Tr}(A)\Big]\E\Big[\mathrm{Tr}(B)\Big].
$$
Since the matrix $C$ is independent with respect to $\Phi$ and $D_{N}$ we conclude that 
$$
\E\Big[\mathrm{Tr}\big((X^{*}D_N X)^{-1}\big)\Big]
$$
is equal to
$$
\frac{1}{L}\E\Big[\mathrm{Tr}\big( C^{-2}\big)\Big] \E\Big[\mathrm{Tr}\big((\Phi^{*}D_{N}\Phi)^{-1} \big)\Big].
$$
Let $D_{N}=I_{N}$ then $(X^{*}D_{N}X)^{-1}=(X^{*}X)^{-1}=C^{-2}$ and it is known (see Lemma 2.10 \cite{Ver}) that 
$$
\E\Big[\mathrm{Tr}\big( C^{-2}\big)\Big]=\frac{L}{N-L}.
$$
Putting all the pieces together we see that  
\begin{equation}
\mu:=\E\Big[\mathrm{Tr}\big((X^{*}D_N X)^{-1}\big)\Big]=\frac{\det(G)}{\det(\Delta(D_N))}
\end{equation}
finishing the proof.
 \end{proof}

\subsection{Application of the Main Results to $\mathrm{cov}_{L}(D)$}

Using Theorem \ref{main} we can compute several averages over the Stiefel manifold. As an application let us compute 
\begin{equation}\label{cov}
\mathrm{cov}_{L}(D)=\int_{\Omega_{L,M}}{\Phi(\Phi^{*}D\Phi)\Phi^{*}\,d\phi}=\mathrm{diag}(t_1,\ldots,t_M). 
\end{equation}
\noindent First let us note that the same proof in Lemma \ref{lemma} gives us the following result.

\vspace{0.3cm}
\begin{lemma}\label{lemma2}
Let $f$ be a differentiable function in the interval $[d_{min},d_{max}]$ where $d_{min}=\min\{d_{i}\}$ and $d_{max}=\max\{d_{i}\}$. Then the following formula holds, 
\begin{equation}
\frac{\partial}{\partial d_{k}}
\int_{\Omega_{L,N}} {\mathrm{Tr}\Big(f(\Phi^{*}D_N\Phi)\Big)\,\,d\phi}
\end{equation}
is equal to
$$
\Bigg[\int_{\Omega_{L,N}}{\Phi f'(\Phi^{*}D_N\Phi)\Phi^{*}\,\,d\phi}\Bigg]_{kk}
$$
where $A_{kk}$ means the $(k,k)$ entry of the matrix $A$.
\end{lemma}

\vspace{0.4cm}
\noindent Let $D$ an $M\times M$ diagonal matrix of rank $N$ and let $L\leq N$. By equation (\ref{cov}) and Lemma \ref{lemma2} we know that
$$
t_{k}=\frac{1}{2}\cdot\frac{\partial}{\partial d_{k}}\int_{\Omega_{L,M}}{\mathrm{Tr}\Big((\Phi^{*}D\Phi)^{2}\Big)\,d\phi}.
$$
\noindent Since 
\begin{eqnarray*}
\int_{\Omega_{L,M}}{\mathrm{Tr}\Big((\Phi^{*}D\Phi)^{2}\Big)\,d\phi} & = & \frac{L}{M}\cdot\frac{L+1}{M+1}s_{(2,0)}(D)\\ 
& - & \frac{L}{M}\cdot\frac{L-1}{M-1}s_{(1,1)}(D)
\end{eqnarray*}
\noindent where 
$$
s_{(2,0)}(D)=\sum_{i=1}^{M}{d_{i}^{2}}+\sum_{i<j}{d_{i}d_{j}}
$$ 
and 
$$
s_{(1,1)}(D)=\sum_{i<j}{d_{i}d_{j}}.
$$
\noindent Then 
$$\frac{\partial s_{(2,0)(D)}}{\partial d_{k}}=\mathrm{Tr}(D)+d_{k}\quad\text{and}\quad\frac{\partial s_{(1,1)(D)}}{\partial d_{k}}=\mathrm{Tr}(D)-d_{k}.$$

\noindent Therefore, $t_{k}$ is equal to
$$
\frac{1}{2}\cdot\frac{L}{M}\Bigg[\frac{L+1}{M+1}\Big(d_{k}+\mathrm{Tr}(D)\Big)-\frac{L-1}{M-1}\Big(\mathrm{Tr}(D)-d_{k}\Big) \Bigg],
$$
hence
$$
t_{k} =\frac{L}{M}\Bigg[ \frac{M-L}{M^{2}-1}\mathrm{Tr}(D) + \frac{ML-1}{M^{2}-1}d_{k}\Bigg]. 
$$
Therefore,
\begin{eqnarray*}
\mathrm{cov}_{L}(D) & = & \int_{\Omega_{L,M}}{\Phi(\Phi^{*}D\Phi)\Phi^{*}\,d\phi}\\
& = &\frac{L}{M}\Bigg[ \frac{M-L}{M^{2}-1}\mathrm{Tr}(D)\cdot I_{M} + \frac{ML-1}{M^{2}-1}\cdot D\Bigg].
\end{eqnarray*}

\subsection{Application of the Main Results to $\mathrm{invcov}_{L}(D)$}
\par As we mention in the introduction and subsequent Sections the problem we are mainly interested is to compute $\mathrm{invcov}_{L}(D)$ when $D$ is an $M\times M$ diagonal matrix of rank $N$ and $L\leq N$. Let $D=\mathrm{diag}(D_{N},0_{M-N})$ where $D_{N}=\mathrm{diag}(d_{1},\ldots,d_{N})$. As we previously saw $\mathrm{invcov}_{L}(D)$ is a diagonal matrix that can be decomposed as 
$$
\mathrm{invcov}_{L}(D)=\mathrm{diag}(\Lambda_{L}(D_{N}),\mu I_{M-N})
$$
where 
$$
\mathrm{\Lambda}_{L}(D_{N})=\int_{\Omega_{L,N}}{\Phi(\Phi^{*}D_N\Phi)^{-1}\Phi^{*}\,d\phi}=\mathrm{diag}(\lambda_1,\ldots,\lambda_N)
$$
and 
$$
\mu=\mathrm{Tr}\Big[\E\big((X^{*}D_{N}X)^{-1}\big)\Big]
$$ 
where the average is taken over the ensemble of all $N\times L$ Gaussian random matrices $X$ with independent and complex entries with zero mean and unit variance.

 \vspace{0.2cm}
\noindent Using Lemma \ref{lemma2} we see that
$$
\lambda_{k}=\frac{\partial}{\partial d_{k}}\int_{\Omega_{L,N}}{\mathrm{Tr}\log(\Phi^{*}D_N\Phi)\,d\phi}.
$$
where 
$$
\int_{\Omega_{L,N}}{\mathrm{Tr}\log(\Phi^{*}D_N\Phi)\,d\phi}
$$
can be explicitly computed using Theorem \ref{main}. On the other hand, Proposition \ref{mainprop} gives us an explicit formula for $\mu$. Therefore, given $D$ we obtained close form expressions for all the entries of the matrix $\mathrm{invcov}_{L}(D)$ for every $M,N$ and $L$.

\subsection{Small Dimension Example}

\vspace{0.3cm}
In this subsection we will compute $\mathrm{invcov}_{L}(D)$ for a small dimensional example. Let $M=4$, $N=2$ and $L=1$ and
$D=\mathrm{diag}(d_1,d_2,0,0)$. Hence following our previous notation $D_{2}=\mathrm{diag}(d_1,d_2)$. Let us first consider the case $d_1=d_2=d$. In this case using equation (\ref{eq_d0}) and equation (\ref{eq_d1}) we see that
$$
\mathrm{invcov}_{1}(D)=\mathrm{diag}(d^{-1}/2,d^{-1}/2,d^{-1},d^{-1}).
$$
The more interesting case is $d_1\neq d_2$. Applying Theorem \ref{main} and Lemma \ref{lemma2} we see that
$$
\mathrm{invcov}_{1}(D)=\mathrm{diag}(\lambda_1,\lambda_2,\mu,\mu)
$$
where
\begin{eqnarray*}
\lambda_{i} & = & \frac{\partial}{\partial d_{i}}\int_{\Omega_{1,2}}{\mathrm{Tr}\log(\Phi^{*}D_2\Phi)\,d\phi}\\
& = & \frac{\partial}{\partial d_{i}} \Bigg(\frac{\mathrm{det}(H)}{\mathrm{det}(\Delta(D_2))}\Bigg)
\end{eqnarray*}
where 
$$
\Delta(D_{2}) = 
\Bigg[\begin{array}{ll}
d_1 & d_2 \\ 
1 & 1
\end{array}\Bigg]
$$
and 
$$
H = 
\Bigg[\begin{array}{ll}
d_{1}\log d_1-d_1 & d_{2}\log d_{2}-d_2  \\ 
1 & 1
\end{array}\Bigg].
$$
Doing the calculations we see that 
\begin{equation}
\lambda_{1}=\frac{d_{2}\log d_{2}-d_{2}\log d_{1}+d_1-d_2}{(d_1-d_2)^2}
\end{equation} 
and 
\begin{equation}
\lambda_{2}=\frac{d_{1}\log d_{1}-d_{1}\log d_{2}+d_2-d_1}{(d_1-d_2)^2}.
\end{equation} 
On the other hand using Proposition \ref{mainprop}
\begin{equation*}
\mu = \mathrm{Tr}\Big[\E((XD_{2}X^{*})^{-1})\Big]=\frac{\mathrm{det}(G)}{\mathrm{det}(\Delta(D_2))}
\end{equation*}
where 
$$
G = 
\Bigg[\begin{array}{ll}
\log d_1 & \log d_{2}\\ 
1 & 1
\end{array}\Bigg].
$$
Therefore,
\begin{equation}
\mu = \frac{\log d_1-\log d_2}{d_1-d_2}.
\end{equation}

\par To summarize, computation of $\mathrm{invcov}$ is facilitated by first taking the eigenvector and eigenvalue decomposition of the sample covariance matrix and then applying $\mathrm{invcov}$ to the diagonal eigenvalue matrix. At that point there are three alternatives: a straight Monte Carlo expectation, the asymptotic expressions (\ref{mu}) and (\ref{lambda_asymp}), or the closed form analytical expressions according to Proposition \ref{mainprop} and Theorem \ref{main}.

\section{Preliminaries on the Limit of Large Random Matrices}

\vspace{0.4cm}
\par A random matrix is a measurable map $X$, defined on some probability space $(\Omega,P)$ and which takes values in a matrix algebra, $M_{n}(\C)$ say.  In other words, $X$ is a matrix whose entries are (complex) random variables on $(\Omega,P)$. One often times identifies $X$ with the probability measure $X(P)$ it induces on $M_{n}(\C)$ and forgets about the underlying space $(\Omega,P)$. Random matrices appear in a variety of mathematical fields and in physics too. For a more complete and detailed discussion of random matrix limits and free probability see \cite{Voi2}, \cite{Voi1}, \cite{Speicher}, \cite{Speicher2} and \cite{Ver}. Here we will review only some key points.

\par Let us consider a sequence $\{A_{N}\}_{N\in\mathbb{N}}$ of self--adjoint $N\times N$ random matrices $A_{N}$. In which sense can we talk about the limit of these matrices? It is evident that such a limit does not exist as an $\infty\times\infty$ matrix and there is no convergence in the usual topologies. What converges and survives in the limit are the moments of the random matrices. Let $A=(a_{ij}(\omega))_{i,j=1}^{N}$ where the entries $a_{ij}$ are random variables on some probability space $\Omega$ equipped with a probability measure $P$. Therefore,
\begin{equation}
 \mathbb{E}(\mathrm{tr_{N}}(A_{N})):=\frac{1}{N}\sum_{i=1}^{N} \int_{\Omega}{a_{ii}(\omega)\,dP(\omega)}
\end{equation}
and we can talk about the $k$--th moment $\mathbb{E}(\mathrm{tr_{N}}(A_{N}^{k}))$ of our random matrix $A_{N}$, and it is well known that for nice random matrix ensembles these moments converge for $N\to\infty$. So let us denote by $\alpha_{k}$ the limit of the $k$--th moment,
\begin{equation}
\alpha_{k}:=\lim_{N\to\infty}{\mathbb{E}(\mathrm{tr_{N}}(A_{N}^{k}))}.
\end{equation}
Thus we can say that the limit consists exactly of the collection of all these numbers $\alpha_{k}$. However, instead of talking about a collection of numbers we prefer to identify these numbers as moments of some random variable $A$. Now we can say that our random matrices $A_{N}$ converge to a variable $A$ in distribution (which just means that the moments of $A_{N}$ converge to the moments of $A$). We will denote this by $A_{N}\to A$.

\par One should note that for a self--adjoint $N\times N$ matrix $A=A^{*}$, the collection of moments corresponds also to a probability measure $\mu_{A}$ on the real line, determined by 
\begin{equation}
 \mathrm{tr}_{N}(A^{k})=\int_{\mathbb{R}}{t^{k}\,d\mu_{A}(t)}.
\end{equation}
This measure is given by the eigenvalue distribution of $A$, i.e. it puts mass $\frac{1}{N}$ on each of the eigenvalues of $A$ (counted with multiplicity):
\begin{equation}
 \mu_{A}=\frac{1}{N}\sum_{i=1}^{N}{\delta_{\lambda_{i}}},
\end{equation}
where $\lambda_{1},\ldots,\lambda_{N}$ are the eigenvalues of $A$. In the same way, for a random matrix $A$, $\mu_{A}$ is given by the averaged eigenvalue distribution of $A$. Thus, moments of random matrices with respect to the averaged trace contain exactly that type of information in which one is usually interested when dealing with random matrices.

\vspace{0.3cm}
\begin{example} Let us consider the basic example of random matrix theory. Let $G_{N}$ be an $N\times N$ self-adjoint random matrix whose upper--triangular entries are independent zero-mean random variables with variance $\frac{1}{N}$ and fourth moments of order $O(\frac{1}{N^2})$. Then the famous Theorem of Wigner can be stated in our language as
\begin{equation}
 G_{N}\to s, \hspace{0.4cm}\text{where $s$ is a semicircular random variable},
\end{equation}
where semicircular just means that the measure $\mu_{s}$ is given by the semicircular distribution (or, equivalently, the even moments of the
variable $s$ are given by the Catalan numbers).
\end{example}

\vspace{0.3cm}
\begin{example}
Another important example in random matrix theory is the Wishart ensemble. Let $X$  be an $N\times K$ random matrix whose entries are independent zero-mean random variables with variance $\frac{1}{N}$ and fourth moments of order $O(\frac{1}{N^2})$. As $K,N\to\infty$ with $\frac{K}{N}\to\beta$,   
\begin{equation}
 X^{*}X\to x_{\beta},
\end{equation}
where $x_\beta$ is a random variable with the Marcenko--Pastur law $\mu_{\beta}$ with parameter $\beta$.
\end{example}
\vspace{0.3cm}

\par The empirical cumulative eigenvalue distribution function of an $N\times N$ self--adjoint random matrix $A$ is defined by the random function 
$$
F_{A}^{N}(\omega, x):=\frac{\#\{k\,:\,\lambda_{k}\leq x\}}{N}
$$
where $\lambda_{k}$ are the (random) eigenvalues of $A(\omega)$ for each realization $\omega$. For each $\omega$ this function determines a probability measure $\mu_{N}(\omega)$ supported on the real line. These measures $\{\mu_{N}(\omega)\}_{\omega}$ define a Borel measure $\mu_{N}$ in the following way. Let $B\subset\mathbb{R}$ be a Borel subset then
$$
\mu_{N}(B):=\mathbb{E}\Big(\mu_{N}(\omega)(B)\Big).
$$

\par A new and crucial concept, however, appears if we go over from the case of one variable to the case of more variables. 

\vspace{0.3cm}
\begin{definition}
Consider $N\times N$ random matrices $A_{N}^{(1)},\ldots,A_{N}^{(m)}$ and variables $A_{1},\ldots,A_{m}$ (living in some abstract algebra $\mathcal{A}$ equipped with a state $\varphi$). We say that 
\begin{equation*}
 (A_{N}^{(1)},\ldots,A_{N}^{(m)})\to (A_{1},\ldots,A_{m})
\end{equation*}
in distribution if and only if 
\begin{equation}
 \lim_{N\to\infty}{\mathbb{E}\Big(\mathrm{tr_{N}}\Big(A_{N}^{(i_1)}\cdots A_{N}^{(i_k)}\Big)\Big)}=\varphi(A^{(i_1)}\cdots A^{(i_k)})
\end{equation}
for all choices of $k$, $1\leq i_{1},\ldots, i_{k}\leq m$.
\end{definition}

\vspace{0.3cm}
The $A_{1},\ldots,A_{m}$ arising in the limit of random matrices are a priori abstract elements in some algebra $\mathcal{A}$, but it is good to know
that in many cases they can also be concretely realized by some kind of creation and annihilation operators on a full Fock space. Indeed,
free probability theory was introduced by Voiculescu for investigating the structure of special operator algebras generated by these type of
operators. In the beginning, free probability had nothing to do with random matrices.

\vspace{0.3cm}
\begin{example}
Let us now consider the example of two independent Gaussian random matrices $G_{N}^{(1)},G_{N}^{(2)}$ (i.e., each of them is a self--adjoint
Gaussian random matrix and all entries of $G_N^{(1)}$ are independent from all entries of $G_{N}^{(2)}$ ). Then one knows that all joint moments converge, and we can say that 
\begin{equation}
(G_{N}^{(1)},G_{N}^{(2)})\to (s_1,s_2),
\end{equation} 
where Wigner's Theorem tells us that both $s_1$ and $s_2$ are semicircular. The question is: What is the relation between $s_1$ and $s_2$? Does the independence between $G_N^{(1)}$ and $G_{N}^{(2)}$ survive in some form also in the limit? The answer is yes and is
provided by a basic theorem of Voiculescu which says that $s_1$ and $s_2$ are free. For a formal definition of freeness and more results in free probability see \cite{Voi2}, \cite{Voi1}, \cite{Speicher} and \cite{Speicher2}.
\end{example}

\section{Asymptotic results}\label{asymp}

\par In practical application when the values of $M$ and $N$ are too large the expressions of the previous Section could be difficult to evaluate and we need simpler mathematical formulas. It is for this reason that in this Section we will provide asymptotic results for the entries of $\mathrm{invcov}_{L}(D)$. As we previously saw $\mathrm{invcov}_{L}(D)$ is equal to 
\begin{equation*}
\Bigg[\begin{array}{ll}
\E(X^{*}(XD_{N}X^{*})^{-1}X) & 0 \\ 
0 & \E(Y^{*}(XD_{N}X^{*})^{-1}Y)\end{array}\Bigg]
\end{equation*}
\noindent where $D=\mathrm{diag}(D_{N},0_{M-N})$ and $D_{N}$ an $N\times N$ diagonal matrix of full rank. In this Section we will get asymptotic results for both terms $\E(X^{*}(XD_{N}X^{*})^{-1}X)$ and $\E(Y^{*}(XD_{N}X^{*})^{-1}Y)$ as $N,L\to\infty$ with $\lim_{N\to\infty}\frac{N}{L}=\beta$.

\vspace{0.3cm}
\noindent Recall that $\mu>0$ was defined as 
\begin{equation}
\mu=\mathrm{Tr}\Big[\E((XD_{N}X^{*})^{-1})\Big].
\end{equation}
As we already observed, if $X$ and $Y$ are Gaussian independent random matrices it is clear that 
\begin{eqnarray*}
\E(Y^{*}(XD_{N}X^{*})^{-1}Y) & = & \mathrm{Tr}\Big[\E((XD_{N}X^{*})^{-1})\Big]I_{M-N}\\
& = & \mu I_{M-N}.
\end{eqnarray*}
\noindent Let us introduce the $\eta$--transform of an $N\times N$ random matrix $A$ as, 
$$
\eta_A(\gamma):=\frac{1}{N}\mathrm{Tr}\Big[\E((1+\gamma A)^{-1})\Big]\quad \text{for}\quad \gamma>0.
$$
\noindent Analogously, given $\nu$ a probability measure on $\R$ we can define the $\eta$--transform of $\nu$ as 
$$
\eta_{\nu}(\gamma)=\int_{\R}{\frac{1}{1+\gamma t}d\nu(t)}.
$$
\noindent It is not difficult to see that 
$$
\lim_{\gamma\to\infty}{\eta_{\nu}(\gamma)}=\nu(\{0\})
$$ 
and 
$$
\lim_{\gamma\to\infty}{\gamma\eta_{\nu}(\gamma)}=\int_{\R}{t^{-1}d\nu(t)}.
$$ 
Consider $D=(D_{N})_{N\in\N}$ a collection of diagonal $N\times N$ diagonal matrices such that $D_{N}$ converge in distribution to a probability measure $\nu$, i.e. for all $k\geq 0$ 
$$
\lim_{N\to\infty}{\frac{1}{N}\E\Big(\mathrm{Tr}(D_{N}^{k})\Big)}=\int_{0}^{\infty}{t^{k}\,d\nu(t)}.
$$
\noindent Let $(X_{N})_{N\in\N}$ be a sequence of $L\times N$ complex Gaussian random matrices with standard i.i.d. entries. Then by Theorem 2.39 in \cite{Ver} we know that  $\frac{1}{L}X_{N}D_{N}X_{N}^{*}$ converges in distribution (moreover, almost surely) as as $N,L\to\infty$ with $\lim_{N\to\infty}\frac{N}{L}=\beta$ to a probability distribution whose $\eta$--transform satisfies 
\begin{equation}\label{eta}
\beta(1-\eta_{\nu}(\gamma\eta))=1-\eta(\gamma) 
\end{equation}
\noindent where $\eta_{\nu}$ is the $\eta$--transform of the limit in distribution of $D=(D_{N})_{N\in\N}$. Note that $\lim_{\gamma\to\infty}{\eta(\gamma)}=0$ and $\lim_{\gamma\to\infty}{\gamma\eta(\gamma)}=\mu$. Therefore, taking limit as $\gamma$ goes to infinity on both sides of equation (\ref{eta}) we get 
\begin{equation}\label{eta2}
\eta_{\nu}(\mu)=\frac{\beta-1}{\beta}
\end{equation}

\noindent since by definition and the convergence in distribution for $N$ sufficiently large we have that
\begin{equation}\label{eq-eta}
\eta_{\nu}(\mu)\approx\frac{1}{N}\mathrm{Tr}\Big[\E(1+\mu D_{N})^{-1}\Big] 
\end{equation}
and
$$
\frac{N-L}{N}\approx\frac{\beta-1}{\beta}.
$$
\noindent The symbol $\approx$ denotes that equality holds in the limit.

\vspace{0.3cm}
\noindent Using equation (\ref{eta2}) and equation (\ref{eq-eta}) we obtain
\begin{equation}\label{mu}
\frac{N-L}{N}\approx \frac{1}{N}\sum_{k=1}^{N}{\frac{1}{1+\mu d_{k}}}
\end{equation}

\noindent where $D_{N}=\mathrm{diag}(d_{1},\ldots,d_{N})$. Note that this equation implicitly and uniquely defines $\mu>0$.  

\vspace{0.3cm}

\begin{remark}
\begin{enumerate}
	\item If $N=L$ then $\mu=\infty$ (which is not surprising).
	\item If $d_{1}=\ldots=d_{N}=\alpha$ then $\mu=\alpha^{-1}\mathrm{Tr}\,\E((XX^{*})^{-1})$. Using Lemma 2.10 from \cite{Ver} we know that $\mathrm{Tr}\,\E((XX^{*})^{-1})=\frac{L}{N-L}$. Therefore for this case
	\begin{equation}\label{eq_d1}
	\mu=\frac{L}{\alpha(N-L)}
	\end{equation}
	which agrees exactly (without the approximation) with equation (\ref{mu}). 
\end{enumerate}
 
\end{remark}

\par Now we will compute the asymptotics of the other, and more complicated, term $$\Lambda_{L}(D_{N})=\E(X^{*}(XD_{N}X^{*})^{-1}X)=\mathrm{diag}(\lambda_{1},\ldots,\lambda_{N}).$$

\vspace{0.3cm}
\begin{lemma}\label{lemma}
Let $f$ be a differentiable function on the interval $[d_{min},d_{max}]$ where $d_{min}=\min\{d_{i}\}$ and $d_{max}=\max\{d_{i}\}$. Then the following formula holds 
\begin{equation}
\frac{\partial}{\partial d_{k}}\E\Big[\mathrm{Tr}(f(XDX^{*}))\Big]=\E\Big[X^{*}f'(XDX^{*})X\Big]_{kk}. 
\end{equation}
\end{lemma}

\vspace{0.4cm}

\begin{proof}
It is enough to prove this Lemma for the case $f$ is a polynomial and by linearity it is enough to prove it for the case $f(x)=x^{n}$. Let $E_{k}$ be the $N\times N$ matrix whose entries are all zero except the entry $E_{kk}$ which is equal to 1. Note that $\frac{\partial}{\partial d_{k}}(D)=E_{k}$. If $f(x)=x^{n}$ then $f(XDX^{*})=(XDX^{*})^{n}$ and for each $k$ it is easy to see that
$$
\frac{\partial}{\partial d_{k}}(XDX^{*})^{n}=\sum_{i=0}^{n-1}{(XDX^{*})^{i}XE_{k}X^{*}(XDX^{*})^{n-(i+1)}}.
$$
Therefore, 
\begin{eqnarray*}
\mathrm{Tr}\Bigg( \frac{\partial}{\partial d_{k}}f(XDX^{*})\Bigg) & = & n\mathrm{Tr}\Big( XE_{k}X^{*}(XDX^{*})^{n-1}\Big)\\
& = & \mathrm{Tr}\Big( E_{k}X^{*}f'(XDX^{*})XE_{k}\Big)\\
& = & \Big( X^{*}f'(XDX^{*})X\Big)_{kk}
\end{eqnarray*}
where we use the trace property and the fact that $E_{k}^{2}=E_{k}$. Taking expectation on both sides we finish the proof.
\end{proof}
\vspace{0.3cm}
\noindent As a corollary we obtain.
\vspace{0.3cm}
\begin{cor} If 
$$
\Lambda_{L}(D_{N})=\E(X^{*}(XD_{N}X^{*})^{-1}X)=\mathrm{diag}(\lambda_{1},\ldots,\lambda_{N})
$$
then $\lambda_{k}=\frac{\partial}{\partial d_{k}}\,\E\Big[\mathrm{Tr\,\,log}(XD_{N}X^{*}) \Big]$. 
\end{cor}

\vspace{0.3cm}
\par Before continuing let us define the Shannon transform of a probability distribution. Given $\nu$ a probability distribution supported in $\R$ we define 
$$\vartheta_{\nu}(\gamma)=\int_{\R}{\log(1+\gamma t)d\nu}.$$

\noindent It is easy to see that $\lim_{\gamma\to\infty}{\vartheta_{\nu}(\gamma)-\log(\gamma)}=\int_{\R}{\log(t)d\nu}$.

\vspace{0.3cm}
\noindent Consider, as before, $D=(D_{N})_{N\in\N}$ a collection of diagonal $N\times N$ diagonal matrices such that $D_{N}$ converge in distribution to a probability measure $\nu$. Let $(X_{N})_{N\in\N}$ be a sequence of $L\times N$ complex Gaussian random matrices with standard i.i.d. entries. Then by equation 2.167 in \cite{Ver} we know that  $\frac{1}{L}X_{N}D_{N}X_{N}^{*}$ converges in distribution (moreover, almost surely) as as $N,L\to\infty$ with $\lim_{N\to\infty}\frac{N}{L}=\beta$ to a probability distribution $\eta_{\beta}$ whose Shannon--transform $\vartheta$ satisfies 
\begin{equation}\label{Shannon}
\vartheta(\gamma)=\beta \vartheta_{\nu}(\gamma\eta(\gamma))-\log\eta(\gamma)+\eta(\gamma)-1.
\end{equation}

\noindent Subtracting $\log(\gamma)$ on both sides and taking the limit $\gamma\to\infty$ we obtain
\begin{equation}
\int_{0}^{\infty}{\log(t)\,d\nu_{\beta}}=\beta \vartheta_{\nu}(\mu)-\log(\mu) -1. 
\end{equation}

\noindent For $N$ and $L$ sufficiently large 
\begin{equation}
\int_{0}^{\infty}{\log(t)\,d\nu_{\beta}}=\lim_{L\to\infty}{\frac{1}{L}\E\Big[\mathrm{Tr\,\,log}\Big(\frac{1}{L}X_{N}D_{N}X_{N}^{*}\Big) \Big]}
\end{equation}
$$
\approx\frac{1}{L}\Big(\E\Big[\mathrm{Tr\,\,log}{(X_{N}D_{N}X_{N}^{*})}\Big] -L\log(L)\Big)\nonumber
$$
\noindent and 
\begin{eqnarray}
\beta\vartheta_{\nu}(\mu) & = & \beta\lim_{N\to\infty}{\frac{1}{N}\E\Big[\mathrm{Tr\,\,log}(1+\mu D_{N}) \Big]}\\
& \approx & \frac{\beta}{N}\sum_{k=1}^{N}{\log(1+\mu d_{k})}. \nonumber
\end{eqnarray}
\noindent Hence, for $N$ sufficiently large 
$$\E\Big[\mathrm{Tr\,\,log}(X_{N}D_{N}X_{N}^{*})\Big]\approx \sum_{k=1}^{N}{\log(1+\mu d_{k})}+L\log\Big(\frac{L}{e\mu}\Big).$$

\noindent Taking partial derivatives on both sides with respect to $d_{k}$ we obtain
\begin{equation}\label{lambda_asymp}
\lambda_{k} \approx \frac{\partial\mu}{\partial d_{k}}\cdot\sum_{i=1}^{N}{\frac{d_i}{1+\mu d_i}} + \frac{\mu}{1+\mu d_{k}} - \frac{L}{\mu}\frac{\partial \mu}{\partial d_{k}}. 
\end{equation}

\section{Simulations}\label{simulations}
\par In this Section we present some of the simulations performed. Let $\Sigma$ be a $100\times 100$ true covariance matrix with Toeplitz with entries $\Sigma_{i,j}=\exp(-\frac{1}{3}|i-j|)$. Assume we take $N=50$ measurements and we want recover $\Sigma$ to the best of our knowledge. After performing the measurements we construct the sample covariance matrix $K_{x}$. In Figure \ref{Fig1} and \ref{Fig2} we can see the eigenvalues of the true matrix, the eigenvalues of the sample covariance matrix (raw data) and the eigenvalues of the new matrix after performing the randomization to the sample covariance for different values of $L$. In other words we are comparing the eigenvalues of $\Sigma$, $K_{x}$ and of $\mathrm{invcov}_{L}(K_{x})^{-1}$ for different values of $L$. Recall that $\mathrm{invcov}_{L}(K_{x})$ is an estimate for $\Sigma^{-1}$ and therefore $\mathrm{invcov}_{L}(K_{x})^{-1}$ is an estimate for $\Sigma$.
We also performed the same experiment for the case that the true covariance matrix is a multiple of the identity. This corresponds to the case in which the sensors are uncorrelated. See Figure \ref{Fig3} and \ref{Fig4}. 
\begin{figure}[!Ht]
  \begin{center}
    \includegraphics[width=3.4in, height=4.0in]%
{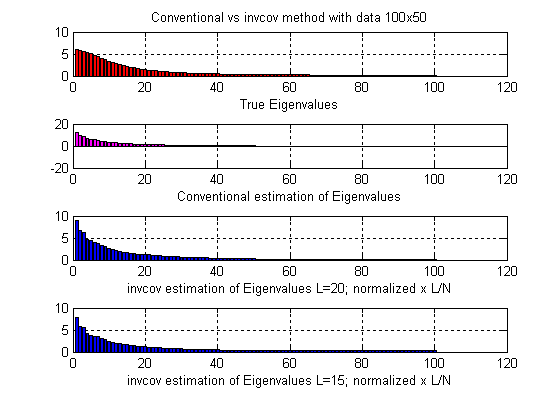}
    \caption{Comparison of the eigenvalues of the true covariance matrix and sample covariance matrix vs. $\mathrm{invcov}$ estimate. The conventional eigenvalue estimate gives 50 eigenvalues that are precisely zero, which is not correct. Our $\mathrm{invcov}$ estimate is much closer to the actual eigenvalue distribution.}
    \label{Fig1}
  \end{center}
\end{figure}

\begin{figure}[!Ht]
  \begin{center}
    \includegraphics[width=3.4in, height=4.0in]%
{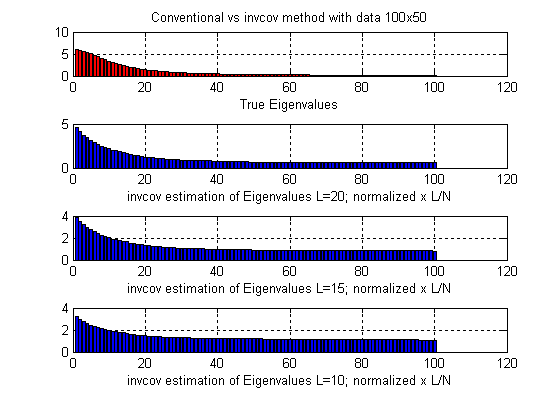}
    \caption{Comparison of the eigenvalues of the true covariance matrix and sample covariance matrix vs. $\mathrm{invcov}$ estimate}
    \label{Fig2}
  \end{center}
\end{figure}

\begin{figure}[!Ht]
  \begin{center}
    \includegraphics[width=3.4in, height=4.0in]%
{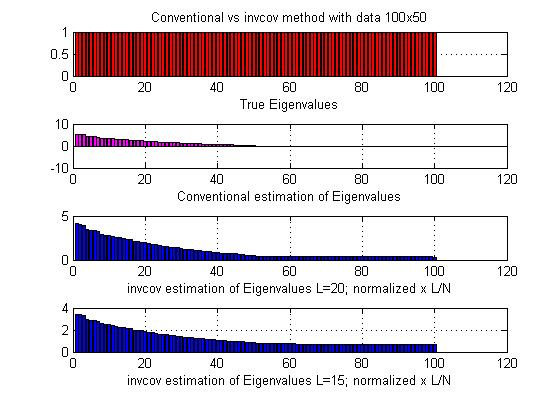}
    \caption{Comparison of the eigenvalues of the true covariance matrix and sample covariance matrix vs. $\mathrm{invcov}$ estimate}
    \label{Fig3}
  \end{center}
\end{figure}

\begin{figure}[!Ht]
  \begin{center}
    \includegraphics[width=3.4in, height=4.0in]%
{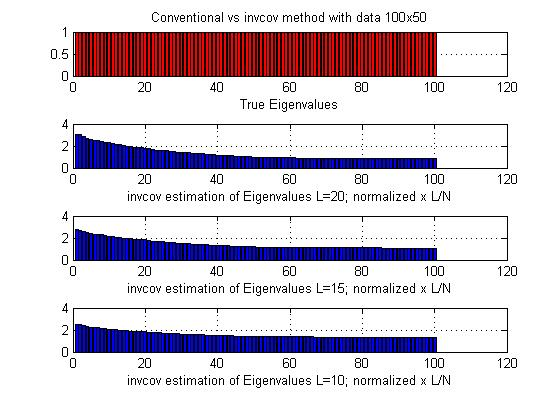}
    \caption{Comparison of the eigenvalues of the true covariance matrix and sample covariance matrix vs. $\mathrm{invcov}$ estimate}
    \label{Fig4}
  \end{center}
\end{figure}



\vspace{0.3cm}
\par Let $A$ be an $N\times N$ matrix we define the Frobenius norm as $\norm{A}_{2}=\frac{1}{N^{2}}\mathrm{Tr}(A^{*}A)$. In Figure \ref{Fig7}, \ref{Fig8} and \ref{Fig9} we compare the performance of our estimator with the more standard estimator of Ledoit and Wolf \cite{LW}. The experiment was performed with a $60\times 60$ true covariance matrix $\Sigma$  with Toeplitz entries $\Sigma_{i,j}=\exp(-\frac{1}{\beta}|i-j|)$ and $N=30$. We compute the Frobenius norm $\norm{\Sigma-\mathrm{invcov}_{L}(K_{x})^{-1}}_{2}$ for the different values of $L$ and we compare it with $\norm{\Sigma-K_{LW}}_{2}$. We use $\beta=10, 50$ and $100$. We can see that our method outperform Ledoit and Wolf for a big range of the parameter $L$. We also note that in the three cases $L=20$ is the optimum.

\begin{figure}[!Ht]
  \begin{center}
    \includegraphics[width=3.4in, height=3.2in]%
{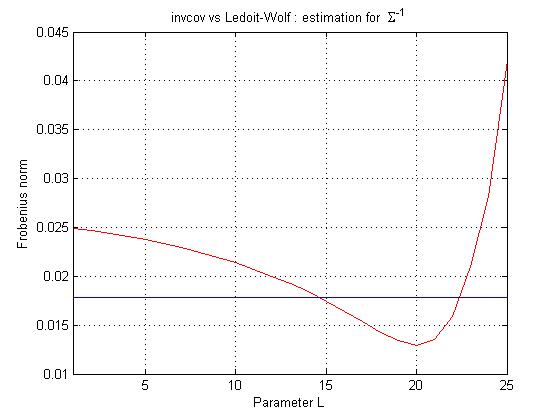}
    \caption{Performance comparison with respect to Ledoit and Wolf for $\beta=10$. The horizontal blue line is the performance of the LW estimate, our $\mathrm{invcov}$ estimate outperforms LW for $L$ between 15 and 22.}
    \label{Fig7}
  \end{center}
\end{figure}

\begin{figure}[!Ht]
  \begin{center}
    \includegraphics[width=3.4in, height=3.2in]%
{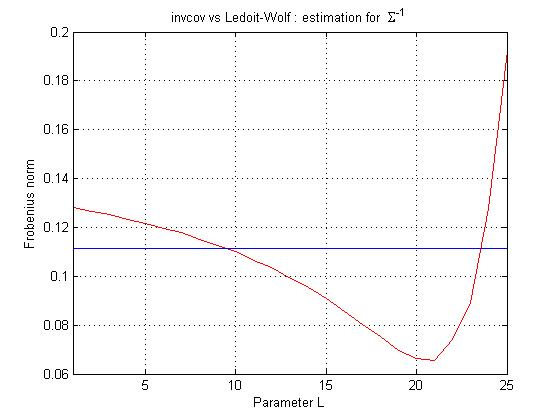}
    \caption{Performance comparison with respect to Ledoit and Wolf for $\beta=50$}
    \label{Fig8}
  \end{center}
\end{figure}

\begin{figure}[!Ht]
  \begin{center}
    \includegraphics[width=3.4in, height=3.2in]%
{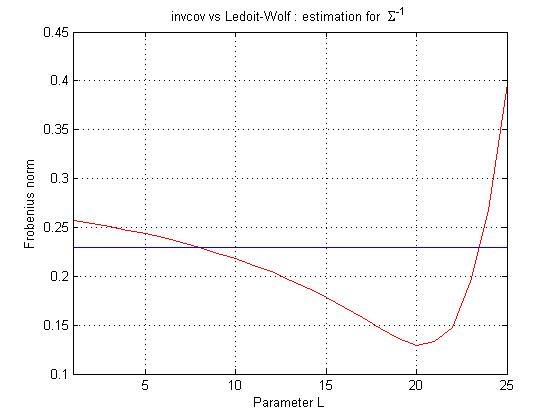}
    \caption{Performance comparison with respect to Ledoit and Wolf for $\beta=100$}
    \label{Fig9}
  \end{center}
\end{figure}

\section{Comments}
Our investigation indicates that the new method presented in this article is interesting and promising from the 
mathematical point of view, as well as the practical one. Even though we were able to find the asymptotics formulas 
as well as close form expressions for $\mathrm{invcov}_{L}$, and more generally for $\mathrm{fcov}_{L}$, both estimates for $\Sigma^{-1}$
and $f(\Sigma)$, there is a natural question still unanswered. How to find the optimum $L$? For instance assume we
know that the matrix $\Sigma$ is an $M\times M$ Toeplitz distributed with entries 
$\Sigma_{i,j}=\exp(-\frac{1}{\beta}|i-j|)$ with unknown $\beta$. How does the optimum $L$ depends on
$M$ and $N$? In the simulations presented in Section \ref{simulations} we see that for $M=60$, $N=30$ the optimum $L$ is equal
to 20 and it seems to be independent on the value of $\beta$. We believe that this a question interesting to explore.

\section{Acknowledgments}
We would like to thank Thomas Bengtsson and Yuliy Baryshnikov for many helpful discussions and comments.

\appendices
\section{Average over the permutation matrices}
In this Appendix we study what happen if we substitute the average over the isotropically random unitary matrices by the set of permutation matrices. Let $D$ be an $N\times N$ diagonal matrix with $D=\mathrm{diag}(d_{1},\ldots,d_{N})$ and $L\leq N$. We will denote $e_{p}$ the orthonormal vector $1\times N$ in $\C^{N}$ whose entries are zero except at the $p^{\,th}$--coordinate that is 1. A permutation matrix is a $L\times N$ matrix of the form
$$\Phi=\begin{pmatrix}
           e_{p_1} \\
           e_{p_2} \\
           \vdots \\
           e_{p_{L}} \\
         \end{pmatrix}\hspace{0.7cm} \text{where} \hspace{0.4cm} p_{i}\neq p_{j}\hspace{0.4cm} \text{for}\hspace{0.4cm} i\neq j.$$

\vspace{0.3cm}
\begin{theorem} Let $D$ be a diagonal matrix as before. Then
$$\E\,\Big[\Phi^{*}(\Phi D \Phi^{*})^{-1}\Phi\,\,:\,\, \Phi\,\,\,\,L\times N\,\, \mathrm{permutation}\Big ]=\frac{L}{N}\,D^{-1}.$$
\end{theorem}

\vspace{0.3cm}
\begin{proof}
\noindent We will first show that 
$$
\Phi D \Phi^{*}=\mathrm{diag}(\,d_{p_1},d_{p_2},\ldots,d_{p_{L}}).
$$ 
Given $1\leq p\leq N$ then $e_{p}\cdot D=d_{p}\,e_{p}$. Therefore,
\begin{eqnarray*}
\Phi D \Phi^{*} & = & \left(
                    \begin{array}{c}
                      d_{p_1}e_{p_1}\\
                      d_{p_2}e_{p_2}\\
                      \vdots \\
                      d_{p_{L}}e_{p_{L}} \\
                    \end{array}
                  \right)\left(
                           \begin{array}{cccc}
                             e_{p_1}^{T} & e_{p_2}^{T} & \ldots & e_{p_L}^{T} \\
                           \end{array}
                         \right)\\ 
& = & \mathrm{diag}(\,d_{p_1},d_{p_2},\ldots,d_{p_{L}}).
\end{eqnarray*}

\noindent Hence,
$$
(\Phi D \Phi^{*})^{-1}=\mathrm{diag}(\,d_{p_1}^{-1},d_{p_2}^{-1},\ldots,d_{p_{L}}^{-1}).
$$

\noindent Thus, 
\begin{equation}
\Phi^{*}(\Phi D \Phi^{*})^{-1}\Phi=\sum_{i=1}^{L}{d_{p_i}^{-1}e_{p_i}^{T}e_{p_i}}.
\end{equation}

\noindent Note that $E_{p}:=e_{p}^{T}e_{p}$ is the matrix that has all entries zero except at the $(p,p)$ entry which is 1. The total number of different permutation matrices is $\frac{N!}{(N-L)!}$. Therefore,
$$\E\,[\Phi^{*}(\Phi D \Phi^{*})^{-1}\Phi]=\frac{(N-L)!}{N!}\cdot\sum_{(p_1,\ldots,p_{L})}{\Bigg [ \sum_{i=1}^{L}{d_{p_i}^{-1} E_{p_i}}\Bigg ]}$$

\noindent where the first sum is running over all the possible permutation matrices. The number of permutation matrices that have $e_{p}$ as one of their rows is $\frac{(N-1)!L!}{(L-1)!(N-L)!}$. Hence,
\begin{eqnarray*}
\E[\Phi^{*}(\Phi D \Phi^{*})^{-1}\Phi] & = & \frac{(N-L)!}{N!}\frac{(N-1)!L!}{(L-1)!(N-L)!}\cdot D^{-1}\\
& = & \frac{L}{N}\,D^{-1}.
\end{eqnarray*}
\end{proof}

\end{document}